\documentclass[reqno,a4paper,draft]{amsart}

\usepackage{enumitem}
\setenumerate{label=\textnormal{(\arabic*)}}

\usepackage{amsmath,amssymb,dsfont,verbatim,bm,array,mathtools}
\usepackage[latin1]{inputenc}
\usepackage{booktabs}
\usepackage[raggedright]{titlesec}
\usepackage{mathtools}

\titleformat{\chapter}[display]
{\normalfont\huge\bfseries}{\chaptertitlename\\thechapter}{20pt}{\Huge}
\titleformat{\section}
{\normalfont\Large\bfseries\center}{\thesection}{1em}{}
\titleformat{\subsection}
{\normalfont\large\bfseries}{\thesubsection}{1em}{}
\titleformat{\subsubsection}[runin]
{\normalfont\normalsize\bfseries}{\thesubsubsection}{1em}{}
\titleformat{\paragraph}[runin]
{\normalfont\normalsize\bfseries}{\theparagraph}{1em}{}
\titleformat{\subparagraph}[runin]
{\normalfont\normalsize\bfseries}{\thesubparagraph}{1em}{}
\titlespacing*{\chapter} {0pt}{50pt}{40pt}
\titlespacing*{\section} {0pt}{3.5ex plus 1ex minus .2ex}{2.3ex plus .2ex}
\titlespacing*{\subsection} {0pt}{3.25ex plus 1ex minus .2ex}{1.5ex plus .2ex}
\titlespacing*{\subsubsection}{0pt}{3.25ex plus 1ex minus .2ex}{1.5ex plus .2ex}
\titlespacing*{\paragraph} {0pt}{3.25ex plus 1ex minus .2ex}{1em}
\titlespacing*{\subparagraph} {\parindent}{3.25ex plus 1ex minus .2ex}{1em}

\input xypic
\xyoption{all}


\newtheorem{theorem}{Theorem}[section]

\theoremstyle{definition}

\theoremstyle{remark}
\newtheorem{remark}[theorem]{Remark}

\DeclareMathOperator{\Jac}{Jac}

\DeclareMathOperator{\pd}{pd}

\begin{document}
\title{A slight generalization of Keller's theorem}
\author{Vered Moskowicz}
\address{Department of Mathematics, Bar-Ilan University, Ramat-Gan 52900, Israel.}
\email{vered.moskowicz@gmail.com}
\thanks{The author was partially supported by an Israel-US BSF grant 2010/149}

\begin{abstract}
The famous Jacobian problem asks: Is a morphism 
$f:\mathbb{C}[x,y]\to \mathbb{C}[x,y]$
having an invertible Jacobian, invertible?
If we add the assumption that
$\mathbb{C}(f(x),f(y))=\mathbb{C}(x,y)$, 
then $f$ is invertible; 
this result is due to O. H. Keller (1939).
We suggest the following slight generalization of Keller's theorem: 
If $f:\mathbb{C}[x,y]\to \mathbb{C}[x,y]$ 
is a morphism having an invertible Jacobian,
and if there exist $n \geq 1$, 
$a \in \mathbb{C}(f(x),f(y))^*$ and 
$b \in \mathbb{C}(f(x),f(y))$ such that 
$(ax +b)^n \in \mathbb{C}(f(x),f(y))$,
then $f$ is invertible.
A similar result holds for $\mathbb{C}[x_1,\ldots,x_m]$.
\end{abstract}

\maketitle                  

\section{Introduction}
Throughout this paper we work over $\mathbb{C}$
(Theorems \ref{main thm} and \ref{main thm 2})
or just over a field of characteristic zero 
(Theorems \ref{thm special case}, \ref{thm special case 2},
\ref{thm add result} and \ref{thm add result}). 
Some of the known results that we recall in the preliminaries section
are true over a PID or a UFD or even any ring; 
sometimes $2 \neq 0$.
However, in view of \cite[Lemma 1.1.14]{book van den essen},
it seems that $\mathbb{C}$ is good enough.
Notice that in our main theorem, Theorem \ref{main thm}, 
we apply Formanek's theorem
\cite[Theorem 2]{formanek} which is over $\mathbb{C}$.
Even if Formanek's theorem is still valid over a more general ring, 
$\mathbb{R}$ is not enough for our main theorem,
since we wish to have all $n$'th roots of unity 
for some $n$ which is probably $\geq 3$.   

In all our theorems:
$k=\mathbb{C}$ or $k$ is a field of characteristic zero,
$f:k[x,y] \to k[x,y]$ 
is a morphism that satisfies
$\Jac(f(x),f(y)) \in k^*$,
and 
$p:= f(x)$, $q:= f(y)$.
More generally, when $m \geq 2$:
$f:k[x_1,\ldots,x_m] \to k[x_1,\ldots,x_m]$ 
is a morphism that satisfies
$\Jac(f(x_1),\ldots,f(x_m)) \in k^*$,
and
$r_1:=f(x_1),\ldots,r_m:=f(x_m)$.
Keller's theorem says that the birational case has a positive answer: 
If 
$\mathbb{C}(r_1,\ldots,r_m)= \mathbb{C}(x_1,\ldots,x_m)$, 
then $f$ is invertible.
It can be found in Keller's paper \cite{keller}, van den Essen's book 
\cite[Corollary 1.1.35]{book van den essen} 
and Bass-Connell-Wright's paper \cite[Theorem 2.1 (a) iff (b)]{bass connell wright}.
When $m=2$, Keller's theorem says that if 
$\mathbb{C}(p,q)= \mathbb{C}(x,y)$, 
then $f$ is invertible.
We bring our observation that it is possible to slightly generalize Keller's theorem: 
Instead of demanding 
$\mathbb{C}(x,y)= \mathbb{C}(p,q)$, 
it is enough to demand that
$\mathbb{C}(p,q)((ax+ b)^n) = \mathbb{C}(p,q)$,
for some 
$n \geq 1$, $a \in \mathbb{C}(p,q)^*$ and 
$b \in \mathbb{C}(p,q)$.
The general case when $m \geq 2$ also seems to be true, namely,
instead of demanding 
$\mathbb{C}(x_1,\ldots,x_m) = \mathbb{C}(r_1,\ldots,r_m)$,
it is enough to demand that
for $m-1$ variables $x_j$
$\mathbb{C}(r_1,\ldots,r_m)((a_j x_j + b_j)^{n_j}) = 
\mathbb{C}(r_1,\ldots,r_m)$,
for some 
$n_j \geq 1$, $a_j \in \mathbb{C}(r_1,\ldots,r_m)^*$ and 
$b_j \in \mathbb{C}(r_1,\ldots,r_m)$.
 
\begin{remark}
Of course, instead of taking 
$\mathbb{C} (p,q)((a x + b )^n)$, 
we could have taken
$\mathbb{C} (p,q)((a y + b )^n)$. 
The general case $m \geq 2$ is already in its general form; 
namely, conditions on any $m-1$ variables from $x_1,\ldots,x_m$,
not necessarily the first $m-1$ variables.
\end{remark}

\section{Preliminaries}
We rely on the following nice results: 

\textbf{(1) Wang's and Bass's theorem:}
Let $k$ be a UFD. Then:
$k[x_1,\ldots,x_m] \cap k(r_1,\ldots,r_m)= k[r_1,\ldots,r_m]$.
This result was proved by Wang \cite[Theorem 41(i)]{wang} and generalized by Bass 
for not necessarily polynomial rings
\cite[Remark after Corollary 1.3, page 74]{book bass}
See also 
\cite[Corollary 1.1.34 and Proposition D.1.7]{book van den essen}.
When $m=2$: 
$k[x,y] \cap k(p,q) = k[p,q]$.

\textbf{(2) Equivalent special cases, 
each having a positive answer to the Jacobian Conjecture:} 
Let $k$ be a field of characteristic zero.
Then TFAE:
\begin{itemize}
\item The birational case:
$k(x_1,\dots,x_m) = k(r_1,\ldots,r_m)$.
\item The Galois case:
$k(x_1,\dots,x_m)$ is a Galois extension of $k(r_1,\ldots,r_m)$.
\item The integral case:
$k[x_1,\dots,x_m]$ is integral over $k[r_1,\ldots,r_m]$.
\item The projective case:
$k[x_1,\dots,x_m]$ is a projective $k[r_1,\ldots,r_m]$-module.
\item $f$ is invertible.
\end{itemize}
This and other equivalent conditions can be found in \cite[Theorem 46(iii)]{wang}
and in \cite[Theorem 2.1]{bass connell wright}. 
Some of the above equivalent conditions can be found in 
\cite{keller}(Keller's paper from 1939), \cite{camp}, \cite{abhy}, 
\cite{razar}, \cite{wright}, \cite{oda}, \cite[Theorem 9]{wang2},
\cite[Theorem 2.2.16]{book van den essen}
and probably in other places as well.
 
When $m=2$, TFAE: $k(x,y)=k(p,q)$; $k(x,y)$ is a Galois extension of $k(p,q)$;
$k[x,y]$ is integral over $k[p,q]$; $k[x,y]$ is a projective $k[p,q]$-module. 

\textbf{(3) Formanek's theorem:}
$\mathbb{C}(r_1,\ldots,r_m,x_1,\ldots,x_{m-1})=
\mathbb{C}(x_1,\ldots,x_m)$.
See \cite[Theorem 2]{formanek}.
When $m=2$:
$\mathbb{C}(p,q,x)=\mathbb{C}(x,y)$.
(Of course, we also have 
$\mathbb{C}(p,q,y)=\mathbb{C}(x,y)$. 
Similarly, the equality in the general case is true adjoining any 
$m-1$ variables).

\textbf{(4) Hamann's theorem:} 
If $p$ and $q$ are monic in $y$,
then $k[x,y]$ is integral over $k[p,q][x]$.
(having the same field of fractions: $k(x,y)=k(p,q)(x)$).
See \cite[Proposition 2.1]{haman}.
If $p$ and $q$ of a given morphism $f$
are not monic in $y$, 
then one can consider $gf$ for an appropriate automorphism $g$,
and get $\tilde{p}:=(gf)(x)=g(p)$ and
$\tilde{q}:=(gf)(y)=g(q)$ monic in $y$,
see \cite[Theorem 1.1]{haman}
(``change of variables'').

Although not mentioned explicitly in Hamann's paper, it seems that 
the following is also true: 
If $r_1,\ldots,r_m$ are monic in $x_m$, 
then $k[x_1,\ldots,x_m]$ is integral over 
$k[r_1,\ldots,r_m][x_1,\ldots,x_{m-1}]$. 

 
\begin{remark}
An equivalent way to say $k(x,y)=k(p,q)$ is:
$x \in k(p,q)$ and $y \in k(p,q)$.
Actually, it is enough to demand that $x \in k(p,q)$:
Assume $x \in k(p,q)$.
By Formanek's theorem $k(p,q,x)=k(x,y)$.
Combining the two yields $k(p,q)=k(x,y)$.

Similarly for the general case $m \geq 2$:
It is enough to demand that 
$x_j \in k(r_1,\ldots,r_m)$
for $1 \leq j \leq m-1$
(namely, omitting the demand
$x_m \in k(r_1,\ldots,r_m)$):
By Formanek's theorem 
$k(r_1,\ldots,r_m,x_1,\ldots,x_{m-1})= k(x_1,\ldots,x_m)$,
and combining this with our $m-1$ conditions yields 
$k(r_1,\ldots,r_m)=k(x_1,\ldots,x_m)$.
\end{remark}

\section{The slight generalization}

After the above preparations it is time to prove our main theorem:

\begin{theorem}\label{main thm}
If there exist $n \geq 1$, $a \in \mathbb{C}(p,q)^*$ and 
$b \in \mathbb{C}(p,q)$ such that 
$(ax + b)^n \in \mathbb{C}(p,q)$, 
then $f$ is invertible 
\end{theorem}

Recall that 
$\mathbb{C}(p,q) \subseteq \mathbb{C}(x,y)$
is finite ($x$ and $y$ are algebraic over $\mathbb{C}(p,q)$, 
since $\{p,q,x\}$ and $\{p,q,y\}$ are algebraically dependent over 
$\mathbb{C}$. See \cite[Proposition 1.1.31]{book van den essen}) 
and separable.

\begin{proof}
Denote $A=\mathbb{C}(p,q)$
and $B=\mathbb{C}(x,y)$. 

\begin{itemize}
\item $\mathbb{C}(p,q,(ax+b)^n)= A$ by our assumption.
\item $\mathbb{C}(p,q,ax+b)= \mathbb{C}(p,q,x)= B$;
the first equality is trivial, and the second equality is 
Formanek's theorem mentioned above 
\end{itemize}

We show that $A \subseteq B$ is a Galois extension
or $f$ is birational (which also means that 
$A \subseteq B$ is Galois, trivially), 
hence $f$ is invertible,
as was recalled in the preliminaries section.

Let $h:= T^n-(ax+b)^n \in \mathbb{C}(p,q,(ax+b)^n)[T]$.

There are two options:

\textbf{First option:} $h$ is separable.
So $\mathbb{C}(p,q,ax+b)$ is the splitting field of the separable polynomial
$T^n-(ax+b)^n \in \mathbb{C}(p,q,(ax+b)^n)[T]$.
Therefore $A \subseteq B$ is Galois (see \cite[Definition 4.47]{rowen}).

\textbf{Second option:} $h$ is not separable. 
Hence $\epsilon_u(ax+b)=\epsilon_v(ax+b)$,
where $\epsilon_u \neq \epsilon_v$ are two different 
$n$'th roots of $1$.
Then from
$(\epsilon_u-\epsilon_v)(ax+b)=0$,
we get that
$ax+b=0$.
So $x= -b/a \in \mathbb{C}(p,q)$,
and we get 
$\mathbb{C}(p,q)=\mathbb{C}(p,q,ax+b)=
\mathbb{C}(p,q,x)=\mathbb{C}(x,y)$.
Therefore we arrived at the birational case.
\end{proof}

It seems that a similar result holds for 
$\mathbb{C}[x_1,\ldots,x_m]$:

\begin{theorem}\label{main thm 2}
If there exist $n_1,\ldots,n_{m-1} \geq 1$, 
$a_1,\ldots,a_{m-1} \in \mathbb{C}(r_1,\ldots,r_m)^*$ and 
$b_1\ldots,b_{m-1} \in \mathbb{C}(r_1,\ldots,r_m)$
such that for each $1 \leq j \leq m-1$
$(a_j x_j + b_j)^{n_j} \in \mathbb{C}(r_1,\ldots,r_m)$,
then $f$ is invertible.
\end{theorem}
 
Recall that 
$\mathbb{C}(r_1,\ldots,r_m) \subseteq \mathbb{C}(x_1,\ldots,x_m)$
is finite (for every $1 \leq i \leq m$, 
$x_i$ is algebraic over $\mathbb{C}(r_1,\ldots,r_m)$, 
since $\{r_1,\ldots,r_m,x_i\}$ are algebraically dependent over 
$\mathbb{C}$. See \cite[Proposition 1.1.31]{book van den essen}) 
and separable.
 
\begin{proof}
Denote $A=\mathbb{C}(r_1,\ldots,r_m)$
and $B=\mathbb{C}(x_1,\ldots,x_m)$. 

\begin{itemize}
\item $\mathbb{C}(r_1,\ldots,r_m,(a_1 x_1+b_1)^{n_1},\ldots,
(a_{m-1} x_{m-1}+b_{m-1})^{n_{m-1}})= A$ 
by our assumptions.

\item 
$\mathbb{C}(r_1,\ldots,r_m,a_1 x_1+b_1,\ldots,a_{m-1}x_{m-1}+b_{m-1})= 
\mathbb{C}(r_1,\ldots,r_m,x_1,\ldots,x_{m-1})=B$;
the first equality is trivial, and the second equality is Formanek's theorem
mentioned above 
\end{itemize}

We show that $A \subseteq B$ is a Galois extension
or $f$ is birational, hence $f$ is invertible. 

Let $h:= \prod_j h_j \in 
\mathbb{C}(r_1,\ldots,r_m,(a_1 x_1+b_1)^{n_1},\ldots,
(a_{m-1} x_{m-1}+b_{m-1})^{n_{m-1}})[T]$,
where $h_j:= T^{n_j}-(a_jx_j+b_j)^{n_j}$,
$1 \leq j \leq m-1$.

There are two options:

\textbf{First option:} $h$ is separable.
So
$\mathbb{C}(r_1,\ldots,r_m,a_1 x_1+b_1,\ldots,a_{m-1}x_{m-1}+b_{m-1})$  
is the splitting field of the separable polynomial
$h \in 
\mathbb{C}(r_1,\ldots,r_m,(a_1 x_1+b_1)^{n_1},\ldots,
(a_{m-1} x_{m-1}+b_{m-1})^{n_{m-1}})[T]$.
Therefore $A\subseteq B$ is Galois.

\textbf{Second option:} $h$ is not separable. 
Multiple roots of $h$ can be of the following two forms:
\begin{itemize}
\item $h_{j_0}$ is not separable:
There exists $1 \leq j_0 \leq m-1$
such that 
$\epsilon_u(a_{j_0}x_{j_0}+b_{j_0})=\epsilon_v(a_{j_0}x_{j_0}+b_{j_0})$,
where $\epsilon_u \neq \epsilon_v$ are two different 
$n_{j_0}$'th roots of $1$.
Then from
$(\epsilon_u-\epsilon_v)(a_{j_0}x_{j_0}+b_{j_0})=0$,
we get that
$a_{j_0}x_{j_0}+b_{j_0}=0$.
So $x_{j_0}= -b_{j_0}/a_{j_0} \in \mathbb{C}(r_1,\ldots,r_m)$.

\item $h_{i_0}$ and $h_{j_0}$ are not relatively prime:
There exist $1 \leq i_0 \neq j_0 \leq m-1$
such that
$\epsilon(a_{i_0}x_{i_0}+b_{i_0})=\delta(a_{j_0}x_{j_0}+b_{j_0})$,
where $\epsilon^{n_{i_0}}=1$ and
$\delta^{n_{j_0}}=1$.
Then $x_{j_0} \in \mathbb{C}(r_1,\ldots,r_m,x_{i_0})$.
\end{itemize}

In each form we get that $x_{j_0}$ is a redundant generator of 

$\mathbb{C}(r_1,\ldots,r_m,x_1,\ldots,x_{m-1})=
\mathbb{C}(r_1,\ldots,r_m,a_1 x_1+b_1,\ldots,a_{m-1}x_{m-1}+b_{m-1})$.
After removing all the redundant $x_j$'s,
we get one of the two following options:
\begin{itemize}
\item $\mathbb{C}(r_1,\ldots,r_m,x_1,\ldots,x_{m-1})=
\mathbb{C}(r_1,\ldots,r_m,x_{t_1},\ldots,x_{t_s})$,
where 
$1 \leq t_1 < \ldots < t_s \leq m-1$
are such that 
$h_{t_1},\ldots,h_{t_s}$ are separable and relatively prime
Then 
$B= \mathbb{C}(r_1,\ldots,r_m,a_{t_1}x_{t_1}+b_{t_1},\ldots,a_{t_s}x_{t_s}+b_{t_s})$
is the splitting field of the separable polynomial
$h_{t_1}\cdots h_{t_s} \in A[T]$,
so $A \subseteq B$ is Galois.  

\item There are no such $h_{t_1},\ldots,h_{t_s}$.
Then for each $1 \leq j \leq m-1$,
$x_j \in \mathbb{C}(r_1,\ldots,r_m)$,
hence $A=B$, the birational case.
\end{itemize}
\end{proof}

Inspired by \cite[Theorem 46]{wang}, \cite[Theorem 2.1]{bass connell wright} 
and \cite[Theorem 2.2.16]{book van den essen},
we wish to ask the following question:
Is it possible to find a weaker condition than 
$\mathbb{C}(p,q) \subseteq \mathbb{C}(x,y)$ 
being a Galois extension which is equivalent to 
$\mathbb{C}(p,q)(a x + b)^n = \mathbb{C}(p,q)$? 
Since $\mathbb{C}(p,q) \subseteq \mathbb{C}(x,y)$ is finite and separable, 
the weaker condition should be weaker than 
$\mathbb{C}(p,q) \subseteq \mathbb{C}(x,y)$ being normal.

Also:
Is it possible to find a weaker condition than integrality/projectivity of 
$\mathbb{C}[x,y]$ over $\mathbb{C}[p,q]$
which is equivalent to 
$\mathbb{C}(p,q)(a x + b)^n = \mathbb{C}(p,q)$?
Recall that 
$\pd_{\mathbb{C}[p,q]}(\mathbb{C}[x,y]) \in \{0,1\}$ 
\cite[Theorem 44(i)]{wang},
and  
$\pd_{\mathbb{C}[p,q]}(\mathbb{C}[x,y]) =0$ if and only if 
$f$ is invertible \cite[Theorem 46(i)+(iii) (1) iff (9)]{wang}.
Hence, if one can show that 
$\pd_{\mathbb{C}[p,q]}(\mathbb{C}[x,y]) = 1$ implies
$\mathbb{C}(p,q)(a x + b)^n = \mathbb{C}(p,q)$, 
then the Jacobian Conjecture is true.
(Since $\pd_{\mathbb{C}[p,q]}(\mathbb{C}[x,y]) \in \{0,1\}$ 
and each of the two cases implies that $f$ is invertible).


\section{A special case with another proof}

Now we consider a special case in which the following two conditions are assumed:
\begin{itemize}
\item [(1)] $r_1,\ldots,r_m$ are monic in $x_m$.
\item [(2)] Instead of
$a_1,\ldots,a_{m-1} \in \mathbb{C}(r_1,\ldots,r_m)^*$, 
$b_1\ldots,b_{m-1} \in \mathbb{C}(r_1,\ldots,r_m)$,
assume 
$a_1,\ldots,a_{m-1} \in \mathbb{C}^*$, 
$b_1\ldots,b_{m-1} \in \mathbb{C}[r_1,\ldots,r_m]$.
\end{itemize}

On the one hand, this section can be omitted, 
since we already have our above results which are valid 
whether $r_1,\ldots,r_m$ are monic in $x_m$ or not,
and also in our above results there are more options for
$a_j,b_j$. 
On the other hand, we added this section since
we wished to apply Wang's theorem \cite[Theorem 41(i)]{wang} 
which is ture over a more general ring than $\mathbb{C}$,
and we wished to apply Hamann's theorem which is true over a 
field of characteristic zero.
(While, as already mentioned in the introduction,
for the proof of our main theorem \ref{main thm}
we better take $\mathbb{C}$). 
Also, a somewhat similar result, Theorem \ref{thm add result}
(and Theorem \ref{thm add result 2}),
is proved using Wang's and Hamann's theorems,
and we do not know if it can be proved without them.

\begin{theorem}\label{thm special case}
Let $k$ be a field of characteristic zero.
Assume $p$ and $q$ are monic in $y$. 
If there exist $n \geq 1$, $a \in k^*$ and 
$b \in k[p,q]$ such that 
$(ax + b)^n \in k(p,q)$, 
then $f$ is invertible.
\end{theorem}

\begin{remark}
Of course, we could have assumed that $p$ and $q$ are monic in $x$
and demand that
$(ay + b)^n \in k(p,q)$, 
instead of 
$(ax + b)^n \in k(p,q)$. 
\end{remark}

\begin{proof}
Assume $(a x + b)^n \in k(p,q)$
for some $n \geq 1$, $a \in k^*$ and 
$b \in k[p,q]$. 
Expand $(a x +b)^n$ and get that 
$a^n x^n + n a^{n-1}x^{n-1}b+\ldots+b^n \in k(p,q)$. 
Hence, $a^n x^n + n a^{n-1}x^{n-1}b+\ldots+b^n =c/d$, 
for some $c,d \in k[p,q]$ with $d \neq 0$.
Multiply by $a^{-n}$ and get that 
$x^n + n a^{-1}x^{n-1}b+\ldots+a^{-n}b^n =
a^{-n}c/d :=w$.
$w= x^n + n a^{-1}x^{n-1}b+\ldots+a^{-n}b^n  \in k[x,y]$,
and
$w=a^{-n}c/d \in k(p,q)$,
so, $w \in k[x,y] \cap k(p,q)$.
But $k[x,y] \cap k(p,q) \subseteq k[p,q]$
(this is the non-trivial inclusion of 
$k[x,y] \cap k(p,q) = k[p,q]$), 
hence $w \in k[p,q]$. 
Namely, $x^n + n a^{-1}x^{n-1}b+\ldots+a^{-n}b^n 
\in k[p,q]$, 
which shows that $x$ is integral over $k[p,q]$.
Therefore (by \cite[Corollary 5.2]{book ati} 
and the remark in that page: ``finitely generated algebra+ integral = finitely generated module''), 
$k[p,q] \subseteq k[p,q][x]$ is integral. 
Combine this with Hamann's theorem that 
$k[p,q][x] \subseteq k[x,y]$ is integral, 
and get that $k[p,q] \subseteq k[x,y]$ is integral.
Finally, since the integral case has a positive answer
(as was recalled in the preliminaries section), 
we get that $f$ is invertible.
\end{proof}

And similarly:
\begin{theorem}\label{thm special case 2}
Let $k$ be a field of characteristic zero.
Assume $r_1,\ldots,r_m$ are monic in $x_m$. 
If there exist $n_1,\ldots,n_{m-1} \geq 1$, 
$a_1,\ldots,a_{m-1} \in k^*$ and 
$b_1,\ldots,b_{m-1} \in k[r_1,\ldots,r_m]$ 
such that for each $1 \leq j \leq m-1$,
$(a_jx_j + b_j)^{n_j} \in k(r_1,\ldots,r_m)$, 
then $f$ is invertible.
\end{theorem}

\begin{proof}
Expand $(a_jx_j +b_j)^{n_j}$ and get that 
$(a_j)^{n_j} (x_j)^{n_j} + (n_j) (a_j)^{n_j-1}(x_j)^{n_j-1}b_j+\ldots+(b_j)^{n_j} \in k(r_1,\ldots,r_m)$. 
Hence, 
$(a_j)^{n_j} (x_j)^{n_j} + (n_j) (a_j)^{n_j-1}(x_j)^{n_j-1}b_j+\ldots+(b_j)^{n_j} =c_j/d_j$,
for some $c_j,d_j \in k[r_1,\ldots,r_m]$ with $d_j \neq 0$.
Multiply by $(a_j)^{-n_j}$ and get
(after same considerations as in the proof of Theorem \ref{thm special case})
that $x_j$ is integral over $k[r_1,\ldots,r_m]$.
Therefore, 
$k[r_1,\ldots,r_m] \subseteq k[r_1\ldots,r_m][x_1,\ldots,x_{m-1}]$ 
is integral. 
Combine this with Hamann's theorem that 
$k[r_1,\ldots,r_m][x_1,\ldots,x_{m-1}]
\subseteq k[x_1,\ldots,x_m]$ is integral, 
and get that $k[r_1,\ldots,r_m] \subseteq k[x_1,\ldots,x_m]$ is integral.
Finally, since the integral case has a positive answer, 
we get that $f$ is invertible.
\end{proof}

\begin{remark}
Without knowing Formanek's theorem \cite[Theorem 2]{formanek},
and Hamann's theorem \cite[Proposition 2.1]{haman},
and only knowing Wang's theorem,
we could still have similar results. 
More precisely, instead of demanding one condition ($m-1$ conditions), 
we should have demanded two conditions ($m$ conditions),
because we need to guarantee that each of the $m$ variables 
$x_1,\ldots,x_m$ is integral over $k[r_1,\ldots,r_m]$.
More precisely:
If there exist $n,\tilde{n} \geq 1$, 
$a, \tilde{a} \in k^*$ and 
$b, \tilde{b} \in k[p,q]$ such that
$(ax + b)^n \in k(p,q)$ 
and
$(\tilde{a}y + \tilde{b})^{\tilde{n}} \in k(p,q)$,
then $f$ is invertible.
The proof is not difficult: 
The first condition implies that $x$ is integral over $k[p,q]$, 
and the second condition implies that 
$y$ is integral over $k[p,q]$.
Therefore, $k[x,y]=k[p,q][x,y]$ is integral over $k[p,q]$. 
Hence $f$ is invertible.
The general case $m \geq 2$ is similar.
\end{remark}

\section{An additional result}
The following result is proved quite similarly to 
Theorem \ref{thm special case}.

\begin{theorem}\label{thm add result}
Let $k$ be a field of characteristic zero.
Assume $p$ and $q$ are monic in $y$.  
If there exist $n_1 > n_2 > \ldots > n_l \geq 0$ and
$a_{n_2},\ldots,a_{n_l} \in k[p,q]$ such that
$x^{n_1}+a_{n_2}x^{n_2}+\ldots+a_{n_l}x^{n_l} \in k(p,q)$,
then $f$ is invertible.
\end{theorem}

\begin{proof}
$x^{n_1}+a_{n_2}x^{n_2}+\ldots+a_{n_l}x^{n_l}
\in k(p,q) \cap k[x,y] = k[p,q]$.
Hence, $x$ is integral over $k[p,q]$.
Therefore, $k[p,q] \subseteq k[p,q][x]$ is integral. 
Combine this with Hamann's theorem that 
$k[p,q][x] \subseteq k[x,y]$ is integral, 
and get that $k[p,q] \subseteq k[x,y]$ is integral.
Finally, since the integral case has a positive answer, 
we get that $f$ is invertible.
\end{proof} 

If we remove the assumption that $p$ and $q$ are monic in $y$,
then we should add the condition that
there exist $\tilde{n_1} > \tilde{n_2} > \ldots > \tilde{n_l} \geq 0$ and
$\tilde{a_{\tilde{n_2}}},\ldots,\tilde{a_{\tilde{n_l}}} \in k[p,q]$ 
such that
$y^{\tilde{n_1}}+\tilde{a_{\tilde{n_2}}}y^{\tilde{n_2}}+\ldots+
\tilde{a_{\tilde{n_l}}}y^{\tilde{n_l}} \in k(p,q)$
(in order to get that $y$ is integral over $k[p,q]$. 
So both $x$ and $y$ are integral over $k[p,q]$,
and we arrived at the integral case).

Obviously, the general case result is also true:

\begin{theorem}\label{thm add result 2}
Let $k$ be a field of characteristic zero.
Assume $r_1,\ldots,r_m$ are monic in $x_m$.  
If there exist $n_{1,j} > n_{2,j} > \ldots > n_{l,j} \geq 0$ and
$a_{n_{2,j}},\ldots,a_{n_{l,j}} \in k[r_1,\ldots,r_m]$ such that
for each $1 \leq j \leq m-1$,
$(x_j)^{n_{1,j}}+a_{n_{2,j}}(x_j)^{n_{2,j}}+\ldots+
a_{n_{l,j}}(x_j)^{n_{l,j}} \in k(r_1,\ldots,r_m)$,
then $f$ is invertible.
\end{theorem}

And if we remove the assumption that
$r_1,\ldots,r_m$ are monic in $x_m$,
then we should add the condition that
there exist $n_{1,m} > n_{2,m} > \ldots > n_{l,m} \geq 0$ and
$a_{n_{2,m}},\ldots,a_{n_{l,m}} \in k[r_1,\ldots,r_m]$ such that

$(x_m)^{n_{1,m}}+a_{n_{2,m}}(x_m)^{n_{2,m}}+\ldots+
a_{n_{l,m}}(x_m)^{n_{l,m}} \in k(r_1,\ldots,r_m)$
(in order to get that $x_m$ is integral over $k[r_1,\ldots,r_m]$
So $x_1,\ldots,x_m$ are integral over $k[r_1,\ldots,r_m]$,
and we arrived at the integral case).
\bibliographystyle{plain}

\begin{thebibliography}{00} 

\bibitem{abhy} S. S. Abhyankar, \textit{Expansion techniques in algebraic geometry},
Tata Inst. Fundamental Research, Bombay, 1977.

\bibitem{book ati} M. Atiyah and I. MacDonald, \textit{Introduction to commutative algebra}, Addison-Wesley Publishing Company, 1969. 

\bibitem{bass connell wright} H. Bass, E. Connell and D. Wright, \textit{The Jacobian Conjecture: 
reduction of degree and formal expansion of the inverse}, Bull. Amer. Math. Soc. (New Series) 7 (1982), 287-330.

\bibitem{book bass} H. Bass, Differential structure of Etale extensions of polynomial algebras, 
\textit{Commutative Algebra} (M. Hochster, C. Huneke and J.D. Sally, eds.), Springer-Verlag, New York, 1989, 
Proceedings of a Microprogram Held, June 15-July 2, 1987.

\bibitem{camp} L. A. Campbell, \textit{A condition for a polynomial map to be invertible},
Math. Ann. 205 (1973), 243-248.

\bibitem{book van den essen} A. van den Essen, \textit{Polynomial automorphisms and the Jacobian conjecture}, 
Progress in Mathematics 190, Birkhuser Verlag, Basel, 2000.


\bibitem{formanek} E. Formanek, \textit{Observations about the Jacobian Conjecture},
Houston J. of Math. 20, no.3 (1994), 369-380.

\bibitem{haman} E. Hamann, \textit{Algebraic observations on the Jacobian Conjecture}, J. of Algebra 265, no. 2 (2003),
539-561.

\bibitem{keller} O. H. Keller, \textit{Ganze Cremona-transformationen}, Monatsh. Math. Phys. 47 (1939), 299-306. 



\bibitem{oda} S. Oda, \textit{The Jacobian problem and the simply-connectedness of 
$A^n$ over a field $k$ of characteristic zero}, Osaka Univ. 1980.

\bibitem{razar} M. Razar, \textit{Polynomial maps with constant Jacobian},
Israel J. of Math. 32 (1979), 97-106.

\bibitem{rowen} L. H. Rowen, \textit{Graduate algebra: Commutative view},
Graduate Studies in Mathematics, volume 73, Amer. Math. Soc., 2006.

\bibitem{wang} S. S.-S. Wang, \textit{A Jacobian criterion for separability}, J. of Algebra 65 (1980), 453-494.

\bibitem{wang2} S. S.-S. Wang, \textit{Extension of derivations}, J. of Algebra 69 (1981), 240-246.

\bibitem{wright} D. Wright, \textit{On the Jacobian Conjecture}, Illinois J. of Math. 15, no. 3 (1981), 423-440.

\end{thebibliography}

\end{document}